\documentclass[12pt]{article}
\usepackage{enumerate}
\usepackage{amssymb,amsmath,amsfonts,amsthm,mathrsfs}
\usepackage[colorlinks=true, pdfstartview=FitV, linkcolor=blue, citecolor=blue, urlcolor=blue]{hyperref}
\usepackage{graphicx,color}
\usepackage{cite}
\usepackage{indentfirst}

\allowdisplaybreaks

\begin{document}
\def\rn{{\mathbb R^n}}  \def\sn{{\mathbb S^{n-1}}}
\def\co{{\mathcal C_\Omega}}
\def\z{{\mathbb Z}}
\def\nm{{\mathbb (\rn)^m}}
\def\mm{{\mathbb (\rn)^{m+1}}}
\def\n{{\mathbb N}}
\def\cc{{\mathbb C}}

\newtheorem{defn}{Definition}
\newtheorem{thm}{Theorem}
\newtheorem{lem}{Lemma}
\newtheorem{cor}{Corollary}
\newtheorem{rem}{Remark}

\title{\bf\Large Sharp bound for $m$-linear $n$-dimensional Hardy-Littlewood-P\'{o}lya operator in Morrey space on Heisenberg group
\footnotetext{{\it Key words and phrases}: Hardy Littlewood-P\'{o}lya operator, Hilbert operator, weighted Morrey space, Heisenberg group
\newline\indent\hspace{1mm} {\it 2020 Mathematics Subject Classification}: Primary 42B35; Secondary 26D15.}}

\date{}
\author{Xiang Li, Zhongci Hang, Zhanpeng Gu\footnote{Corresponding author}, Dunyan Yan}
\maketitle
\begin{center}
\begin{minipage}{13cm}
{\small {\bf Abstract:}\quad
In this paper, we obtained the sharp bounds for $m$-linear $n$-dimensional Hardy-Littlewood-P\'{o}lya operator and Hilbert operator in two power weighted Morrey space on Heisenberg group.
 }
\end{minipage}
\end{center}

\section[Introduction]{Introduction}
In \cite{Hardy}, the Hardy-Littlewood-P\'{o}lya operator is defined by
$$
P f(x)=\int_0^{+\infty} \frac{f(y)}{\max (x, y)} d y .
$$
B\'{e}nyi and Oh gave that the norm of Hardy-Littlewood-P\'{o}lya operator on $L^q\left(\mathbb{R}^{+}\right.$)(see \cite{BO}), $1<q<\infty$, was
$$
\|P\|_{L^q\left(\mathbb{R}^{+}\right) \rightarrow L^q\left(\mathbb{R}^{+}\right)}=\frac{q^2}{q-1} .
$$
We had known that the $m$-linear $n$-dimensional Hardy-Littlewood-P\'{o}lya operator as a multilinear generalization of Calder\'{o}n operator, is defined by
$$
T(f_1, \ldots, f_m)(x)=\int_{\mathbb{R}^{n m}} \frac{f_1\left(y_1\right) \cdots f_m\left(y_m\right)}{\max \left(|x|^n,\left|y_1\right|^n, \ldots,\left|y_n\right|^n\right)^m} d y_1 \cdots d y_m ,
$$
where $x \in \mathbb{R}^n \backslash\{0\}$ and $f_1, \ldots, f_m$ are nonnegative locally integrable functions on $\mathbb{R}^n$.

We have learned that the Hilbert operator is the essential extension of the classical Hilbert's inequality, the $m$-linear $n$-dimensional Hilbert operator is defined by
$$
T^*_m(f_1,\ldots, y_m)(x):=\int_{\mathbb{R}^{nm}}\frac{f_1(y_1)\cdots f_m(y_m)}{(|x|^n+|y_1|^n+\cdots+|y_m|^n)^m}d y_1\cdots dy_m,x\in\mathbb{R}^n\backslash\{0\}.
$$
For $m=1$, the following is a known sharp estimate
$$
\int_0^{\infty} T_1^* f(x) g(x) d x \leq \frac{\pi}{\sin (\pi / p)}\|f\|_{L^p(0, \infty)}\|g\|_{L^{p^{\prime}}(0, \infty)}.
$$

In this paper, we will study the sharp bounds of $m$-linear $n$-dimensional Hardy-Littlewood-P\'{o}lya operator and Hilbert operator in two power weighted Morrey space on Heisenberg group. Next, we would like give some concepts and operational rules of Heisenberg group. Anyway, we will provide a detailed proof of the sharp bounds of the above operators.

Firstly, let us recall the basic knowledge of the Heisenberg group. The Heisenberg group $\mathbb{H}^n$ is a non-commutative nilpotent Lie group, the underlying manifold $\mathbb{R}^{2n}\times\mathbb{R}$ with the group law
$$
x \circ y=\left(x_1+y_1, \ldots, x_{2 n}+y_{2 n}, x_{2 n+1}+y_{2 n+1}+2 \sum_{j=1}^n(y_j x_{n+j}-x_j y_{n+j})\right)
$$
and
$$
\delta_r\left(x_1, x_2, \ldots, x_{2 n}, x_{2 n+1}\right)=\left(r x_1, r x_2, \ldots, r x_{2 n}, r^2 x_{2 n+1}\right), \quad r>0,
$$
where $x=(x_1,\cdots,x_{2n},x_{2n+1})$ and $y=(y_1,\cdots,y_{2n},y_{2n+1})$.

The Haar measure on $\mathbb{H}^n$ coincides with the usual Lebesgue measure on $\mathbb{R}^{2n+1}$. We denote the measure of any measurable set $E \subset \mathbb{H}^n$ by $|E|$.Then we have
$$
|\delta_r(E)|=r^Q|E|, d(\delta_r x)=r_Q d x,
$$
where $Q=2n+2$ is called the homogeneous dimension of $\mathbb{H}^n$.

The Heisenberg distance derived from the norm
$$
|x|_h=\left[\left(\sum_{i=1}^{2 n} x_i^2\right)^2+x_{2 n+1}^2\right]^{1 / 4},
$$
where $x=(x_1,x_2,\cdots,x_{2n},x_{2n+1})$ is given by
$$
d(p, q)=d\left(q^{-1} p, 0\right)=\left|q^{-1} p\right|_h.
$$
This distance $d$ is left-invariant in the sense that $d(p,q)$ remains unchanged when $p$ and $q$ are both left-translated by some fixed vector on $\mathbb{H}^n$. Furthermore, $d$ satisfies the triangular inequality \cite{AK}
$$
d(p, q) \leq d(p, x)+d(x, q), \quad p, x, q \in \mathbb{H}^n.
$$

For $r>0$ and $x\in\mathbb{H}^n$, the ball and sphere with center $x$ and radius $r$ on $\mathbb{H}^n$ are given by
$$
B(x, r)=\left\{y \in \mathbb{H}^n: d(x, y)<r\right\},
$$
and
$$
S(x, r)=\left\{y \in \mathbb{H}^n: d(x, y)=r\right\},
$$
respectively. And we have
$$
|B(x, r)|=|B(0, r)|=\Omega_Q r^Q,
$$
where
$$
\Omega_Q=\frac{2 \pi^{n+\frac{1}{2}} \Gamma(n / 2)}{(n+1) \Gamma(n) \Gamma((n+1) / 2)}
$$
is the volume of the unit ball $B(0,1)$ on $\mathbb{H}^n$ and $\omega_Q=Q \Omega_Q$ (see \cite{CT}). For more details about Heisenberg group can be refer to \cite{GB} and \cite{ST}.

In \cite{LX}, Li et.al.defined $m$-linear $n$-dimensional Hardy-Littlewood-P\'{o}lya operator and Hilbert operator on Heisenberg group.
\begin{defn}
	Suppose that $f_1,\ldots,f_m$ be  nonnegative locally integrable functions on $\mathbb{H}^n$. The $m$-linear $n$-dimensional Hardy-Littlewood-P\'{o}lya operator is defined by
	\begin{equation}
		P^h_m(f_1,\ldots,f_m)(x)=\int_{\mathbb{H}^{nm}}\frac{f_1(y_1)\cdots f_m(y_m)}{[\text{max} (|x|_h^Q,|y_1|_h^Q,\ldots,|y_m|_h^Q)]^m}dy_1\cdots dy_m,x\in\mathbb{H}^n\backslash\{0\}.
	\end{equation}
\end{defn}
\begin{defn}
	Suppose that $f_1,\ldots,f_m$ be nonnegative locally integrable functions on $\mathbb{H}^n$. The $m$-linear $n$-dimensional Hilbert operator is defined by
	\begin{equation}
		P^{h*}_m(f_1,\ldots,f_m)(x)=\int_{\mathbb{H}^{nm}}\frac{f_1(y_1)\cdots f_m(y_m)}{(|x|^Q_h+|y_1|_h^Q+\cdots+|y_m|_h^Q)^m}dy_1\cdots d y_m,x\in\mathbb{H}^n\backslash\{0\}.
	\end{equation}
\end{defn}
Next, we begin to consider the form of two power weighted Morrey space on Heisenberg group. We will use the following notation. Any measurable function $\omega$ in a set $E$ is given by
$$
\omega(E)=\int_E\omega dx.
$$
In what follows, $B(|x|_h,R)$ denotes the ball centered at $|x|_h$ with radius $R$, $|B(|x|_h,R)|$ denotes the Lebesgue measure of $B(|x|_h,R)$. We will use these notation in the following definition of two power weighted Morrey spaces.
\begin{defn}
	Let $\omega_1, \omega_2:\mathbb{H}^n\rightarrow(0,\infty)$ are positive measurable function, $1\leq q<\infty$ and $-1/q\leq\lambda<0$. The two power weighted Morrey space $L^{q,\lambda}(\mathbb{H}^n,\omega_1,\omega_2)$ is defined by
	$$
	L^{q, \lambda}(\mathbb{H}^n, \omega_1, \omega_2)=\{f \in L_{l o c}^q:\|f\|_{L^{q, \lambda}(\mathbb{H}^n, \omega_1, \omega_2)}<\infty\},
	$$
	where
	$$
	\|f\|_{L^{q, \lambda}\left(\mathbb{H}^n, w_1, w_2\right)}=\sup _{a \in \mathbb{H}^n, R>0} w_1(B(|a|_h, R))^{-(\lambda+1 / q)}\left(\int_{B(|a|_h, R)}|f(x)|^q w_2(x) d x\right)^{1 / q} .
	$$
\end{defn}
Now, we will give our main results as follows.
\section{Sharp bound for Hardy-Littlewood-P\'{o}lya operator}
\begin{thm}\label{main_1}
	Let $f_i$ be radial functions in $L^{q_j,\lambda}(\mathbb{H}^n,|x|_h^\alpha,|x|^\frac{q_j\gamma_j}{q})$, $1\leq q<\infty,-\frac{1}{q}\leq\lambda<0,1<q_j<\infty,\frac{1}{q}=\frac{1}{q_1}+\cdots+\frac{1}{q_m},\gamma=\gamma_1\cdots+\gamma_m$, $-\frac{1}{q_j} \leq \lambda_j<0 \text { with } j=1, \ldots, m.$
	Then we have
	\begin{equation}
		\|P^h_m(f_1,\ldots,f_m)\|_{L^{q,\lambda}(\mathbb{H}^n,|x|_h^\alpha,|x|_h^\gamma)}\leq A_m\prod_{j=1}^m\|f_j\|_{L^{q_j, \lambda}_j(\mathbb{H}^n,|x|_h^\alpha,|x|_h^{\frac{q_j \gamma_j}{q}})},
	\end{equation}
	where
	\begin{equation}\label{main1}
		A_m=\int_{\mathbb{H}^{nm}}\frac{\prod_{j=1}^{m}|y_j|^{Q\lambda_j-\frac{\gamma_j}{q}+\alpha(\lambda_j+\frac{1}{q_j})}} {[\max(1,|y_1|_h^Q,\ldots,|y_m|_h^Q)]^m}d y_1 \cdots d y_m.
	\end{equation}
	Moreover, if $\alpha\neq -Q,-\frac{1}{q_j}<\lambda_j<0$ and $q\lambda=q_j\lambda_j$ with $j=1,\ldots,m$, then we have
	\begin{equation}
		\|P^h_m(f_1,\ldots,f_m)\|_{\prod_{j=1}^m L^{q_j,\lambda_j}(\mathbb{H}^n,|x|_h^\alpha,|x|_h^{\frac{q_j\gamma_j}{q}})\rightarrow L^{q,\lambda}(\mathbb{H}^n,|x|_h^\alpha,|x|_h^\gamma)}= A_m.
	\end{equation}
\end{thm}
\begin{cor}\label{main_2}
	Let $1\leq q<\infty, -1/q\leq\lambda<0, 1<q_j<\infty, 1/q=1/q_1+\cdots+1/q_m, \gamma=\gamma_1+\cdots+\gamma_m, -1/q_j\leq\lambda_j<0$ with $j=1,\ldots,m$. Then the operator $P^h_m$ is bounded from $\prod_{j=1}^m L^{q_j,\lambda_j}(\mathbb{H}^n, |x|_h^\alpha,|x|_h^{\frac{q_j\gamma_j}{q}})$ to $L^{q,\lambda}(\mathbb{H}^n, |x|_h^\alpha, |x|_h^\gamma)$ if and only if (\ref{main1}) holds. Moreover, if (\ref{main1}) holds, then the following formula holds
	$$
	A_m=\frac{mQ\omega_Q^{m}}{[Q\lambda-\frac{\gamma}{q}+\alpha(\lambda+\frac{1}{q})]\begin{matrix}\prod_{i=1}^m\end{matrix}[Q(1-\lambda_i)+\frac{\gamma_i}{q}-\alpha(\lambda_i+\frac{1}{q_i})]}.
	$$
\end{cor}
For the convenience of proving Theorem \ref{main_1}, we need to give an important lemma.
\begin{lem}\label{main 101}
	Let $1\leq q<\infty$,$-\frac{1}{q}\leq\lambda<0$ and $\alpha,\gamma\in\mathbb{R}$. If $t\in\mathbb{H}^n$ and $f\in L^{q,\lambda}(\mathbb{H}^n,|x|_h^\alpha,|x|_h^\gamma)$ , then we have
	\begin{equation}
		\|f(\delta_{|t|_h} \cdot)\|_{L^{q,\lambda}(\mathbb{H}^n,|x|_h^\alpha,|x|_h^\gamma)}=|t|_h^{Q\lambda-\frac{\gamma}{q}+\alpha(\lambda+\frac{1}{q})}\|f\|_{L^{q,\lambda}(\mathbb{H}^n,|x|_h^\alpha,|x|_h^\gamma)}.
	\end{equation}
	\begin{proof}
		$$
		\begin{aligned}
			& \|f(\delta_{|t|_h} \cdot)\|_{L^{q, \lambda}(\mathbb{H}^n,|x|_h^\alpha,|x|_h^{\gamma})} \\
			=&\sup _{a \in \mathbb{H}^n, R>0}\left(\int_{B(|a|_h, R)}|x|_h^\alpha d x\right)^{-(\lambda+\frac{1}{q})}\left(\int_{B(|a|_h, R)}|f(\delta_{|t|_h} x)|^q|x|_h^\gamma d x\right)^{\frac{1}{q}} \\
			=& \sup _{a \in \mathbb{H}^n, R>0}\left(\int_{B(|a|_h, R)}|x|_h^\alpha d x\right)^{-(\lambda+\frac{1}{q})}\left(\int_{B( |a|_h, R)}|f(\delta _{|t|_h}x)|^q||t|_h x|^\gamma |t|_h^{-\gamma}d x\right)^{\frac{1}{q}} \\
			=&  |t|_h^{-\frac{Q}{q}-\frac{\gamma}{q}}\sup _{a \in \mathbb{H}^n, R>0}\left(\int_{B(|a|_h, R)}|x|_h^\alpha d x\right)^{-(\lambda+\frac{1}{q})}\left(\int_{B(|t\alpha|_h, |t|_h R)}|f(x)|^q|x|_h^\gamma d x\right)^{\frac{1}{q}} \\
			=&  |t|_h^{-\frac{Q}{q}-\frac{\gamma}{q}}\sup _{a \in \mathbb{H}^n, R>0}\left(\int_{B(|a|_h, R)}||t|_hx|_h^\alpha|t|_h^{-\alpha} d x\right)^{-(\lambda+\frac{1}{q})}\left(\int_{B(|t\alpha|_h, |t|_h R)}|f(x)|^q|x|_h^\gamma d x\right)^{\frac{1}{q}} \\
			=&  |t|_h^{Q\lambda-\frac{\gamma}{q}+\alpha(\lambda+\frac{1}{q})}\sup _{a \in \mathbb{H}^n, R>0}\left(\int_{B(|ta|_h, |t|_h R)}|x|_h^\alpha d x\right)^{-(\lambda+\frac{1}{q})}\left(\int_{B(|t\alpha|_h, |t|_h R)}|f(x)|^q|x|_h^\gamma d x\right)^{\frac{1}{q}}\\
			=& |t|_h^{Q\lambda-\frac{\gamma}{q}+\alpha(\lambda+\frac{1}{q})}\|f\|_{L^{q,\lambda}(\mathbb{H}^n,|x|_h^\alpha,|x|_h^\gamma)}.
		\end{aligned}
		$$
		This finishes the proof of Lemma \ref{main 101}.
	\end{proof}
\end{lem}
\begin{proof}[Proof of Theorem \ref{main_1}]
	Inspired by \cite{HQ}, we learn the method that set radial functions. These functions can solve this problem that calculates the sharp bound of $m$-linear $n$-dimensional Hardy-Littlewood-P\'{o}lya operator in two power weighted Morrey space on Heisenberg group.
	First, we set
	$$
	g_j(x)=\frac{1}{\omega_Q} \int_{|\xi_j|_h=1} f_j(\delta_{|x|_h} \xi_j) d\xi_j, \quad x \in \mathbb{H}^n,
	$$
	where $\omega_n=2\pi^{n/2}$ and $j=1,\ldots,m$. Obviously, $g_j$ ($j=1,\ldots,m$) is radical functions. Then we can obtain
	$$
	P_m^h(g_{f_1},\ldots,g_{f_m})(x)=P_m^h(f_1,\ldots,f_m)(x).
	$$
	Using Minkowski's inequality and  H\"{o}lder's inequality, for $j=1,\ldots,m$, we have
	$$
	\begin{aligned}
		&\|g_j\|_{L^{q_j,\lambda_j}(\mathbb{H}^n,|x|_h^\alpha,|x|_h^{\frac{q_j,\gamma_j}{q}})}\\
		=&\frac{1}{\omega_Q}\sup_{a\in\mathbb{H}^n,R>0}\left(\int_{B(|a|_h,R)}|x|_h^\alpha d x\right)^{-(\lambda+\frac{1}{q})}\left(\int_{B(|a|_h,R)}\left|\int_{|\xi_j|_h=1}f_j(\delta_{|x|_h}\xi_j)d \xi_j\right|^q|x|_h^{\frac{q_j\gamma_j}{q}}dx\right)^{\frac{1}{q}}\\
		\leq&\frac{1}{\omega_Q}\sup_{a\in\mathbb{H}^n,R>0}\left(\int_{B(|a|_h,R)}|x|_h^\alpha d x\right)^{-(\lambda+\frac{1}{q})}\int_{|\xi_j|_h=1}\left(\int_{B(|a|_h,R)}|f_j(\delta_{|x|_h}\xi_j)| ^q |x|_h^\frac{q_j\gamma_j}{q} d x \right)^{\frac{1}{q}} d \xi_j\\
		\leq&\left(\int_{B(|a|_h,R)}|x|_h^\alpha dx\right)^{-(\lambda+\frac{1}{q})}\left(\frac{1}{\omega_Q}\int_{|\xi_j|_h=1}\int_{B(|a|_h,R)}|f_j(\delta_{|x|_h}\xi_j)|^q |x|^{\frac{q_j\gamma_j}{q}}dx d \xi_j\right)^{\frac{1}{q}}\\
		=&\|f_j\|_{L^{q_j,\lambda_j}(\mathbb{H}^n,|x|_h^\alpha,|x|_h^{\frac{q_j\gamma_j}{q}})}.
	\end{aligned}
	$$
	Thus we have
	$$
	\frac{\|P^h_m(f_1,\ldots,f_m)\|_{L^{q,\lambda}(\mathbb{H}^n,|x|_h^\alpha,|x|_h^\gamma)}}{\prod_{j=1}^m\|f_j\|_{L^{q_j,\lambda_j}(\mathbb{H}^n,|x|_h^\alpha,|x|_h^{\frac{q_j\gamma_j}{q}})}}
	\leq\frac{\|P^h_m(g_1,\ldots,g_m)\|_{L^{q,\lambda}(\mathbb{H}^n,|x|_h^\alpha,|x|_h^\gamma)}}{\prod_{j=1}^m\|g_j\|_{L^{q_j,\lambda_j}(\mathbb{H}^n|x|_h^\alpha,|x|_h^{\frac{q_j\gamma_j}{q}})}},
	$$
	where implies the operator $P^h_m$ and its restriction to radial functions have the same operator norm in two power weighted Morrey space on Heisenberg group. So, without loss of generality, we assume that $f_j$, $j=1,\ldots,m$, are radial functions in the rest of the proof.
	Then by Minkowski's inequality, H\"{o}lder's inequality and Lemma \ref{main 101}, we can easily get
	$$
	\begin{aligned}
		&\|P_m^h(f_1,\ldots,f_m)\|_{L^{q,\lambda}(\mathbb{H}^n,|x|_h^\alpha,|x|_h^\gamma)}\\
		\leq&\int_{\mathbb{H}^{nm}}\frac{\prod_{j=1}^{m}|y_j|^{Q\lambda_j-\frac{\gamma_j}{q}+\alpha(\lambda_j+\frac{1}{q_j})}} {[\max(1,|y_1|_h^Q,\ldots,|y_m|_h^Q)]^m}d y_1 \cdots d y_m\prod_{j=1}^m\|f_j\|_{L^{q_j, \lambda}(\mathbb{H}^n,|x|_h^\alpha,|x|_h^{\frac{q_j \gamma_j}{q}})}\\
		=&A_m\prod_{j=1}^m\|f_j\|_{L^{q_j, \lambda}(\mathbb{H}^n,|x|_h^\alpha,|x|_h^{\frac{q_j \gamma_j}{q}})}.
	\end{aligned}
	$$
	Taking
	$$
	f_j(x)=|x|_h^{Q\lambda_j-\frac{\gamma_j}{q}+\alpha(\lambda_j+\frac{1}{q_j})}, j=1,\ldots,m,
	$$
	we have
	$$
	\|P_m^h(f_1,\ldots,f_m)\|_{L^{q,\lambda}(\mathbb{H}^n,|x|_h^\alpha,|x|_h^\gamma)}=A_m\prod_{j=1}^m\|f_j\|_{L^{q_j, \lambda}(\mathbb{H}^n,|x|_h^\alpha,|x|_h^{\frac{q_j \gamma_j}{q}})}.
	$$
	Thus, the proof of Theorem \ref{main_1} is finished.
\end{proof}
Now, we give a detailed proof of Corollary \ref{main_2}.
\begin{proof}[Proof of Corollary \ref{main_2}]
	In order to obtain the sharp constant, we learned and borrowed the method that are given by \cite{ZL}. For
	convenience, we take several special cases for separate calculations, the general situation is a summary of these special cases.\\
	\textbf{Case 1 when $m=2$}.\\
	In this case,we have
	$$
	P_2^h=\int_{\mathbb{H}_n}\int_{\mathbb{H}_n}\frac{|y_1|_h^{-Q\lambda_1+\frac{\gamma_1}{q}-\alpha(\lambda_1+\frac{1}{q_1})}|y_2|_h^{-Q\lambda_2+\frac{\gamma_2}{q}-\alpha(\lambda_2+\frac{1}{q_2})}}{[\max(1,|y_1|_h^Q,|y_2|_h^Q)]^2}dy_1dy_2.
	$$
	By calculation, we have
	$$
	\begin{aligned}
		&\int_{\mathbb{H}_n}\int_{\mathbb{H}_n}\frac{|y_1|_h^{-Q\lambda_1+\frac{\gamma_1}{q}-\alpha(\lambda_1+\frac{1}{q_1})}|y_2|_h^{-Q\lambda_2+\frac{\gamma_2}{q}-\alpha(\lambda_2+\frac{1}{q_2})}}{[\max(1,|y_1|_h^Q,|y_2|_h^Q)]^2}dy_1dy_2\\
		=&\int_{|y_1|_h<1}\int_{|y_2|_h<1}|y_1|_h^{-Q\lambda_1+\frac{\gamma_1}{q}-\alpha(\lambda_1+\frac{1}{q_1})}|y_2|_h^{-Q\lambda_2+\frac{\gamma_2}{q}-\alpha(\lambda_2+\frac{1}{q_2})}dy_1dy_2\\
		+&\int_{|y_1|_h>1}\int_{|y_2|_h<|y_1|_h}|y_1|_h^{-Q\lambda_1+\frac{\gamma_1}{q}-\alpha(\lambda_1+\frac{1}{q_1})-2Q}|y_2|_h^{-Q\lambda_2+\frac{\gamma_2}{q}-\alpha(\lambda_2+\frac{1}{q_2})}dy_1dy_2\\
		+&\int_{|y_2|_h>1}\int_{|y_1|_h<|y_2|_h}|y_1|_h^{-Q\lambda_1+\frac{\gamma_1}{q}-\alpha(\lambda_1+\frac{1}{q_1})}|y_2|_h^{-Q\lambda_2+\frac{\gamma_2}{q}-\alpha(\lambda_2+\frac{1}{q_2})-2Q}dy_1dy_2\\
		=&I_0+I_1+I_2.
	\end{aligned}
	$$
	$$
	\begin{aligned}
		I_0=&\int_{|y_1|_h<1}\int_{|y_2|_h<1}|y_1|_h^{-Q\lambda_1+\frac{\gamma_1}{q}-\alpha(\lambda_1+\frac{1}{q_1})}|y_2|_h^{-Q\lambda_2+\frac{\gamma_2}{q}-\alpha(\lambda_2+\frac{1}{q_2})}dy_1dy_2\\
		=&\frac{\omega_Q^2}{[Q(1-\lambda_1)+\frac{\gamma_1}{q}-\alpha(\lambda_1+\frac{1}{q_1})][Q(1-\lambda_2)+\frac{\gamma_2}{q}-\alpha(\lambda_2+\frac{1}{q_2})]},
	\end{aligned}
	$$
	$$
	\begin{aligned}
		I_1=&\int_{|y_1|_h>1}\int_{|y_2|_h<|y_1|_h}|y_1|_h^{-Q\lambda_1+\frac{\gamma_1}{q}-\alpha(\lambda_1+\frac{1}{q_1})-2Q}|y_2|_h^{-Q\lambda_2+\frac{\gamma_2}{q}-\alpha(\lambda_2+\frac{1}{q_2})}dy_1dy_2\\
		=&\frac{\omega_Q}{Q(1-\lambda_2)+\frac{\gamma_2}{q}-\alpha(\lambda_2+\frac{1}{q_2})}\int_{|y_1|_h>1}|y_1|^{-Q(1+\lambda)+\frac{\gamma}{q}-\alpha(\lambda+q)}dy_1\\
		=&\frac{\omega_Q^2}{[-Q\lambda+\frac{\gamma}{q}-\alpha(\lambda+q)][Q(1-\lambda_2)+\frac{\gamma_2}{q}-\alpha(\lambda_2+\frac{1}{q_2})]}.
	\end{aligned}
	$$
	Similarly, we obtain
	$$
	\begin{aligned}
		I_2=&\int_{|y_2|_h>1}\int_{|y_1|_h<|y_2|_h}|y_1|_h^{-Q\lambda_1+\frac{\gamma_1}{q}-\alpha(\lambda_1+\frac{1}{q_1})}|y_2|_h^{-Q\lambda_2+\frac{\gamma_2}{q}-\alpha(\lambda_2+\frac{1}{q_2})-2Q}dy_1dy_2\\
		=&\frac{\omega_Q}{Q(1-\lambda_1)+\frac{\gamma_1}{q}-\alpha(\lambda_1+\frac{1}{q_1})}\int_{|y_2|_h>1}|y_2|^{-Q(1+\lambda)+\frac{\gamma}{q}-\alpha(\lambda+q)}dy_2\\
		=&\frac{\omega_Q^2}{[-Q\lambda+\frac{\gamma}{q}-\alpha(\lambda+q)][Q(1-\lambda_1)+\frac{\gamma_1}{q}-\alpha(\lambda_1+\frac{1}{q_1})]}.
	\end{aligned}
	$$
	Thus we have
	$$
	\begin{aligned}
		&\int_{\mathbb{H}_n}\int_{\mathbb{H}_n}\frac{|y_1|_h^{-Q\lambda_1+\frac{\gamma_1}{q}-\alpha(\lambda_1+\frac{1}{q_1})}|y_2|_h^{-Q\lambda_2+\frac{\gamma_2}{q}-\alpha(\lambda_2+\frac{1}{q_2})}}{[\max(1,|y_1|_h^Q,|y_2|_h^Q)]^2}dy_1dy_2\\
		=&I_0+I_1+I_2\\
		=&\frac{2Q\omega_Q^2}{[-Q\lambda+\frac{\gamma}{q}-\alpha(\lambda+q)][Q(1-\lambda_1)+\frac{\gamma_1}{q}-\alpha(\lambda_1+\frac{1}{q_1})][Q(1-\lambda_2)+\frac{\gamma_2}{q}-\alpha(\lambda_2+\frac{1}{q_2})]}.
	\end{aligned}
	$$
	\textbf{Case 2 when $m\geq 3$}.
	
	Let
	$$
	\begin{aligned}
		&E_0=\{(y_1,\ldots,y_m)\in\mathbb{H}^n\times\cdots\times\mathbb{H}^n:|y_k|_h\leq 1,1\leq k\leq m\};\\
		&E_1=\{(y_1,\ldots,y_m)\in\mathbb{H}^n\times\cdots\times\mathbb{H}^n:|y_1|_h\ge 1,|y_k|_h\leq|y_1|_h,2\leq k\leq m\};\\
		&E_2=\{(y_1,\ldots,y_m)\in\mathbb{H}^n\times\cdots\times\mathbb{H}^n:|y_i|_h\leq 1,|y_j|_h\leq|y_k|_h,|y_k|_h\leq|y_i|_h,1\leq j\leq i\leq k\leq m\};\\
		&E_3=\{(y_1,\ldots,y_m)\in\mathbb{H}^n\times\cdots\times\mathbb{H}^n:|y_m|_h\ge 1,|y_j|_h\leq|y_m|_h,1\leq j\leq m\}.
	\end{aligned}
	$$
	Obviously, we have
	$$
	\bigcup_{j=0}^m E_j=\mathbb{H}^n\times\cdots\times\mathbb{H}^n,E_i\cap E_j=\varnothing.
	$$
	Taking
	$$
	K_j=\int_{\mathbb{H}^{nm}}\frac{\begin{matrix}\prod_{i=1}^m\end{matrix}|y_i|_h^{-Q\lambda_i+\frac{\gamma_i}{q}-\alpha(\lambda_i+\frac{1}{q_i})}}{[\max(1,|y_1|_h^Q,\ldots,|y_m|_h^Q)]^m}dy_1\cdots dy_m,
	$$
	we can calculate that $K_j$ and the above cases are the same with $j=0,1,\ldots,m$.
	$$
	\begin{aligned}
		K_0=&\prod_{i=1}^m\int_{|y_i|\leq 1}|y_i|_h^{-Q\lambda_i+\frac{\gamma_i}{q}-\alpha(\lambda_i+\frac{1}{q_i})}dy_i\\
		=&\frac{\omega_Q^m}{\prod_{i=1}^m[Q(1-\lambda_i)+\frac{\gamma_i}{q}-\alpha(\lambda_i+\frac{1}{q_i})]},
	\end{aligned}
	$$
	$$
	\begin{aligned}
		K_1=&\int_{|y_1|>1}|y_1|_h^{-Q(\lambda_1+m)+\frac{\gamma_1}{q}-\alpha(\lambda_1+\frac{1}{q_1})}dy_1\prod_{i=2}^m\int_{|y_i|_h\leq |y_1|_h}|y_i|_h^{-Q\lambda_i+\frac{\gamma_i}{q}-\alpha(\lambda_i+\frac{1}{q_i})}dy_i\\
		=&\frac{\omega_Q^{m-1}}{\begin{matrix}\prod_{i=2}^m\end{matrix}[Q(1-\lambda_i)+\frac{\gamma_i}{q}-\alpha(\lambda_i+\frac{1}{q_i})]}
		\int_{|y_1|_h>1}|y_1|_h^{-Q(\lambda_1+m)+\frac{\gamma_1}{q}-\alpha(\lambda_1+\frac{1}{q_1})}dy_1\\
		=&\frac{\omega_Q^{m}}{[Q\lambda-\frac{\gamma}{q}+\alpha(\lambda+\frac{1}{q})]\begin{matrix}\prod_{i=2}^m\end{matrix}[Q(1-\lambda_i)+\frac{\gamma_i}{q}-\alpha(\lambda_i+\frac{1}{q_i})]}.
	\end{aligned}
	$$
	So we can deduce that
	$$
	K_j=\frac{\omega_Q^{m}}{[Q\lambda-\frac{\gamma}{q}+\alpha(\lambda+\frac{1}{q})]\begin{matrix}\prod_{1\leq i\leq m,i\neq j}^m\end{matrix}[Q(1-\lambda_i)+\frac{\gamma_i}{q}-\alpha(\lambda_i+\frac{1}{q_i})]}.
	$$
	Then we have obtained that
	$$
	K_m=\frac{mQ\omega_Q^{m}}{[Q\lambda-\frac{\gamma}{q}+\alpha(\lambda+\frac{1}{q})]\begin{matrix}\prod_{i=1}^m\end{matrix}[Q(1-\lambda_i)+\frac{\gamma_i}{q}-\alpha(\lambda_i+\frac{1}{q_i})]}.
	$$
	Combining the above two cases, we have completed the Corollary \ref{main_2}.
\end{proof}
\section{Sharp bound for Hilbert operator}
In this section, we will study the sharp bound of $m$-liner $n$-dimensional Hilbert operator in Morrey space on Heisenberg group. Based on the following introduction, we can easily obtain and prove the following theorem.

Now, we give our main results as follows.
\begin{thm}\label{2}
	Let $f_i$ be radial functions in $L^{q_j,\lambda}(\mathbb{H}^n,|x|_h^\alpha,|x|^\frac{q_j\gamma_j}{q})$, $1\leq q<\infty$, $-\frac{1}{q}\leq\lambda<0$, $1<q_j<\infty$, $\frac{1}{q}=\frac{1}{q_1}+\cdots+\frac{1}{q_m}$, $\gamma=\gamma_1\cdots+\gamma_m$, $-\frac{1}{q_j} \leq \lambda_j<0 \text { with } j=1, \ldots, m.$
	Then we have
	\begin{equation}
		\|P^{h*}_m(f_1,\ldots,f_m)\|_{L^{q,\lambda}(\mathbb{H}^n,|x|_h^\alpha,|x|_h^\gamma)}\leq B_m\prod_{j=1}^m\|f_j\|_{L^{q_j, \lambda}_j(\mathbb{H}^n,|x|_h^\alpha,|x|_h^{\frac{q_j \gamma_j}{q}})},
	\end{equation}
	where
	\begin{equation}\label{2.2}
		B_m=\int_{\mathbb{H}^{nm}}\frac{\prod_{j=1}^{m}|y_j|^{Q\lambda_j-\frac{\gamma_j}{q}+\alpha(\lambda_j+\frac{1}{q_j})}} {1+|y_1|_h^Q+\cdots+|y_m|_h^Q}d y_1 \cdots d y_m.
	\end{equation}
	Moreover, if $\alpha\neq -Q,-\frac{1}{q_j}<\lambda_j<0$, and $q\lambda=q_j\lambda_j$ with $j=1,\ldots,m$, then we have
	\begin{equation}
		\|P^{h*}_m(f_1,\ldots,f_m)\|_{\prod_{j=1}^m L^{q_j,\lambda_j}(\mathbb{H}^n,|x|_h^\alpha,|x|_h^{\frac{q_j\gamma_j}{q}})\rightarrow L^{q,\lambda}(\mathbb{H}^n,|x|_h^\alpha,|x|_h^\gamma)}= B_m.
	\end{equation}
\end{thm}
\begin{cor}\label{main_5}
	Let $1\leq q<\infty$, $-1/q\leq\lambda<0$, $1<q_j<\infty$, $1/q=1/q_1+\cdots+1/q_m$, $\gamma=\gamma_1+\cdots+\gamma_m$, $-1/q_j\leq\lambda_j<0$ with $j=1,\ldots,m$. Then the operator $P^{h*}_m$ is bounded from $\prod_{j=1}^m L^{q_j,\lambda_j}(\mathbb{H}^n, |x|_h^\alpha,|x|_h^{\frac{q_j\gamma_j}{q}})$ to $L^{q,\lambda}(\mathbb{H}^n, |x|_h^\alpha, |x|_h^\gamma)$ if and only if (\ref{2.2}) holds. Moreover, if (\ref{2.2}) holds, then the following formula holds
	$$
	B_m=\left(\frac{2\pi^{n+\frac{1}{2}\Gamma(n/2)}}{(n+1)\Gamma(n)\Gamma((n+1)/2)}\right)^m\frac{\begin{matrix}\prod_{i=1}^m\end{matrix}\Gamma\left(1-\frac{Q\lambda_i-\frac{\gamma_i}{q}+\alpha(\lambda_i+\frac{1}{q_i})}{Q}\right)\Gamma\left(\frac{Q\lambda-\frac{\gamma}{q}+\alpha(\lambda+\frac{1}{q})}{Q}\right)}{\Gamma\left(m\right)}.
	$$
\end{cor}
\begin{proof}[Proof of Theorem \ref{2}]
	This proof is similar to the proof of Theorem \ref{main_1}, we omit the details.
\end{proof}
\begin{proof}[Proof of Corollary \ref{main_5}]
	Now, we will follow the method given by B\'{e}nyi and Oh \cite{BO} to calculate the sharp constant of $m$-linear $n$-dimensional
	Hilbert operator. Using the polar transformations and variables substitution, we have
	$$
	\begin{aligned}
		&\|P_m^{h*}\|_{\begin{matrix}\prod_{i=1}^m\end{matrix}\mathbb{H}_{\alpha_i}^{\infty}(\mathbb{H}^n)\rightarrow \mathbb{H}_{\alpha}^{\infty}(\mathbb{H}^n)}\\
		=&\int_{\mathbb{H}_{nm}}\frac{|y_1|^{-Q\lambda_1+\frac{\gamma_1}{q}-\alpha(\lambda_1+\frac{1}{q_1})}\cdots |y_m|^{-Q\lambda_m+\frac{\gamma_m}{q}-\alpha(\lambda_m+\frac{1}{q_m})}}{(1+|y_1|_h^Q+\cdots+|y_m|_h^Q)}dy_1\ldots dy_m\\
		=&\omega_Q^m\int_0^{\infty}\cdots\int_0^{\infty}\frac{r_1^{Q(1-\lambda_1)+\frac{\gamma_1}{q}-\alpha(\lambda_1+\frac{1}{q_1})-1}\cdots r_m^{Q(1-\lambda_m)+\frac{\gamma_m}{q}-\alpha(\lambda_m+\frac{1}{q_m})-1}}{(1+r_1^Q+\cdots+r_m^Q)^m}dr_1\ldots dr_m\\
		=&\frac{\omega_Q^m}{Q^m}\int_0^{\infty}\cdots\int_0^{\infty}\frac{t_1^{\frac{1}{Q}[Q(1-\lambda_1)+\frac{\gamma_1}{q}-\alpha(\lambda_1+\frac{1}{q_1}]-1}\cdots t_m^{\frac{1}{Q}[Q(1-\lambda_m)+\frac{\gamma_m}{q}-\alpha(\lambda_m+\frac{1}{q_m}]-1}}{(1+t_1+\cdots+t_m)^m}dt_1\ldots dt_m\\
		=&\frac{\omega_Q^m}{Q^m}\int_0^{\infty}\cdots\int_0^{\infty}\frac{t_1^{-Q\lambda_1+\frac{\gamma_1}{q}-\alpha(\lambda_1+\frac{1}{q_1})}\cdots t_m^{-Q\lambda_m+\frac{\gamma_m}{q}-\alpha(\lambda_m+\frac{1}{q_m})}}{(1+t_1+\cdots+t_m)^m}dt_1\ldots dt_m\\
		=&\frac{\omega_Q^m}{Q^m}\int_0^{\infty}\cdots\int_0^{\infty}\frac{t_1^{-\beta_1}\cdots t_m^{-\beta_m}}{(1+t_1+\cdots+t_m)^{Q\lambda-\frac{\gamma}{q}+\alpha(\lambda+\frac{1}{q})}}dt_1\ldots dt_m.
	\end{aligned}
	$$
	Let
	$$
	I_m(Q\lambda-\frac{\gamma}{q}+\alpha(\lambda+\frac{1}{q}),\beta_1,\ldots,\beta_m)=\frac{\omega_Q^m}{Q^m}\int_0^{\infty}\cdots\int_0^{\infty}\frac{t_1^{-\beta_1}\cdots t_m^{-\beta_m}}{(1+t_1+\cdots+t_m)^{Q\lambda-\frac{\gamma}{q}+\alpha(\lambda+\frac{1}{q})}}dt_1\ldots dt_m.
	$$
	By a simple calculate, we have
	$$
	\begin{aligned}
		\int_0^{\infty}\frac{1}{(1+t)^{Q\lambda-\frac{\gamma}{q}+\alpha(\lambda+\frac{1}{q})}t^{\beta}}dt
		&=\int_0^{\infty}(1-t)^{-\beta}t^{Q\lambda-\frac{\gamma}{q}+\alpha(\lambda+\frac{1}{q})+\beta-2}\\
		&=B\left(1-\beta,Q\lambda-\frac{\gamma}{q}+\alpha(\lambda+\frac{1}{q})+\beta-1\right).
	\end{aligned}
	$$
	Using the variables substitution $t_m=(1+t_1+,\cdots+t_m-1)q_m$, we can obtain
	$$
	\begin{aligned}
		&I_m(Q\lambda-\frac{\gamma}{q}+\alpha(\lambda+\frac{1}{q}),\beta_1,\ldots,\beta_m)\\
		=&\int_0^{\infty}\cdots\int_0^{\infty}\frac{t_1^{-\beta_1}\cdots t_{m-1}^{-\beta_{m-1}}}{(1+t_1+\cdots+t_{m-1})^{Q\lambda-\frac{\gamma}{q}+\alpha(\lambda+\frac{1}{q})+\beta_m -1}}dt_1\ldots dt_{m-1}\\
		\times&\int_0^{\infty}\frac{1}{(1+q_m)^{Q\lambda-\frac{\gamma}{q}+\alpha(\lambda+\frac{1}{q})}q_m^{\beta_m}}dq_m\\
		=&B\left(1-\beta_m,Q\lambda-\frac{\gamma}{q}+\alpha(\lambda+\frac{1}{q})+\beta_m-1\right)I_m(Q\lambda-\frac{\gamma}{q}+\alpha(\lambda+\frac{1}{q})+\beta_m -1,\beta_1,\ldots,\beta_m)
	\end{aligned}
	$$
	Combining the inductive method and properties of Gamma function, we have
	$$
	I_m(Q\lambda-\frac{\gamma}{q}+\alpha(\lambda+\frac{1}{q}),\beta_1,\ldots,\beta_m)=\frac{\begin{matrix}\prod_{i=1}^m\end{matrix}\Gamma\left(1-\beta_i\right)
		\Gamma\left(Q\lambda-\frac{\gamma}{q}+\alpha(\lambda+\frac{1}{q})-m+\begin{matrix}\prod_{i=1}^m\end{matrix}\beta_i\right)}{\Gamma\left(Q\lambda-\frac{\gamma}{q}+\alpha(\lambda+\frac{1}{q})\right)}.
	$$
	Thus, we can obtain
	$$
	\begin{aligned}
		&\|P_m^{h*}\|_{\begin{matrix}\prod_{i=1}^m\end{matrix}H_{\alpha_i}^{\infty}(\mathbb{H}^n)\rightarrow H_{\alpha}^{\infty}(\mathbb{H}^n)}\\
		=&\frac{\omega_Q^m}{Q^m}\frac{\begin{matrix}\prod_{i=1}^m\end{matrix}\Gamma\left(1-\frac{Q\lambda_i-\frac{\gamma_i}{q}+\alpha(\lambda_i+\frac{1}{q_i})}{Q}\right)\Gamma\left(\frac{Q\lambda-\frac{\gamma}{q}+\alpha(\lambda+\frac{1}{q})}{Q}\right)}{\Gamma\left(m\right)}\\
		=&\left(\frac{2\pi^{n+\frac{1}{2}\Gamma(n/2)}}{(n+1)\Gamma(n)\Gamma((n+1)/2)}\right)^m\frac{\begin{matrix}\prod_{i=1}^m\end{matrix}\Gamma\left(1-\frac{Q\lambda_i-\frac{\gamma_i}{q}+\alpha(\lambda_i+\frac{1}{q_i})}{Q}\right)\Gamma\left(\frac{Q\lambda-\frac{\gamma}{q}+\alpha(\lambda+\frac{1}{q})}{Q}\right)}{\Gamma\left(m\right)}.
	\end{aligned}
	$$
\end{proof}
Thus, the proof of $Corollary$ \ref{main_5} is finished.
\subsection*{Acknowledgements}
This work was supported by National Natural Science Foundation of  China (Grant No. 12271232) and Shandong Jianzhu University Foundation (Grant No. X20075Z0101).

\begin{flushleft}
	
\vspace{0.3cm}\textsc{Xiang Li\\School of Science\\Shandong Jianzhu University\\Jinan, 250000\\P. R. China}
	
\emph{E-mail address}: \textsf{lixiang162@mails.ucas.ac.cn}

\vspace{0.3cm}\textsc{Zhongci Hang\\School of Science\\Shandong Jianzhu University \\Jinan, 250000\\P. R. China}

\emph{E-mail address}: \textsf{babysbreath4fc4@163.com}

\vspace{0.3cm}\textsc{Zhanpeng Gu\\School of Science\\Shandong Jianzhu University \\Jinan, 250000\\P. R. China}

\emph{E-mail address}: \textsf{guzhanpeng456@163.com}

\vspace{0.3cm}\textsc{Dunyan Yan\\School of Mathematical Sciences\\University of Chinese Academy of Sciences\\Beijing, 100049\\P. R. China}

\emph{E-mail address}: \textsf{ydunyan@ucas.ac.cn}

\end{flushleft}

\end{document}